\documentclass[a4paper]{article}
\usepackage[english]{babel}
\usepackage{pgfplots}\pgfplotsset{compat=1.18}
\usepackage{float}
\usepackage[T1]{fontenc}
\usepackage[utf8]{inputenc}
\usepackage{amsfonts}
\usepackage[affil-it]{authblk}
\usepackage{algpseudocode}
\usepackage{paralist}
\usepackage{caption}
\usepackage{appendix}
\usepackage{verbatim}
\usepackage{wrapfig}
\usepackage{tikz-cd} 
\usepackage{tikz}
\usetikzlibrary{plotmarks}
\usepackage{caption}
\usepackage{amsmath}
\usepackage[arrow, matrix, curve]{xy}
\usepackage{mathtools}

\usepackage[x11names]{xcolor}
\usepackage{amssymb}
\usepackage{wasysym}
\usepackage{listings}
\usepackage[gen]{eurosym}
\usepackage{amssymb,bbm}
\usepackage{stmaryrd}
\usepackage{graphicx}
\usepackage{mathrsfs}
\usepackage{pst-func}
\newcommand{\dd}{\mathrm{d}}

\DeclareMathOperator{\Sym}{\R^{\textit{n}\times \textit{n}}_{\mathrm{Sym}}}
\DeclareMathOperator{\Symtwo}{\R^{2\times 2}_{\mathrm{Sym}}}

\DeclareMathOperator{\tr}{\mathrm{tr}}
\DeclareMathOperator{\Symthree}{\R_{\mathrm{Sym}}^{3\times 3}}
\DeclareMathOperator{\id}{\mathrm{id}}
\DeclareMathOperator{\R}{\mathbb{R}}
\DeclareMathOperator{\N}{\mathbb{N}}
\DeclareMathOperator{\diag}{\mathrm{diag}}
\DeclareMathOperator{\cof}{\mathrm{cof}}

\usepackage{enumitem}
\usepackage{tabularx}
\usepackage{amsthm}

\usepackage{tikz}
\usetikzlibrary{decorations, decorations.text}

\newtheorem{Theorem}{Theorem}[section]
\newtheorem{corollary}[Theorem]{Corollary}
\newtheorem{lemma}[Theorem]{Lemma}
\newtheorem{Bem}[Theorem]{Remark}
\newtheorem{Proposition}[Theorem]{Proposition}
\newtheorem{Definition}[Theorem]{Definition}

\usepackage{fancyhdr}
\usepackage{datetime}
\renewcommand{\sectionmark}[1]{\markboth{~#1}{}}

\fancypagestyle{main}{
\fancyhf{}
\fancyhead[LE]{\footnotesize\nouppercase{\sectionmark}}
\fancyhead[RO]{\footnotesize \nouppercase{\sectionmark}}
\fancyfoot[LE,RO]{\thepage}

}
\renewcommand\subsectionmark[1]{}

\usepackage{csquotes}
\usepackage[backend = biber, sorting=nyt, style =numeric-comp, giveninits=true,
  maxbibnames=99]{biblatex}

\usepackage[breaklinks]{hyperref}

\begin{document}

\title{\bfseries On the symmetric Morrey problem in three dimensions}

%\author[1]{{Georg Dolzmann}}
\author[1]{{Florian Gammer}}

\affil[1]{\footnotesize Faculty of Mathematics, University of Regensburg, Universitätsstraße 31
93053 Regensburg, Germany}

\date{26\textsuperscript{th} October 2026}

\maketitle

\begin{abstract}
We study the relations between seven generalized notions of convexity for continuous integrands in the space of symmetric $3\times 3$ matrices. Here, symmetric quasiconvexity is defined through symmetrized gradients, and symmetric rank-one convexity through symmetrized rank-one lines.
We construct a symmetric quasiconvex integrand that is not polyconvex at symmetric matrices, and an integrand that is both symmetric rank-one convex and Hessian quasiconvex but not symmetric quasiconvex. We also establish a transfer principle that converts smooth rank-one convex functions that lack Hessian quasiconvexity into symmetric rank-one convex integrands while preserving 
this lack of of Hessian quasiconvexity. Applying this scheme to already existing counterexamples yields the existence of symmetric rank-one convex integrands that are not Hessian quasiconvex. Together with known examples, these results determine all pairwise implications among the seven notions considered here.
\end{abstract}

\smallskip


\section{Introduction}

Variational models of geometrically linear elasticity involve
energies of the form
$$\mathbb E(u)=\int_\Omega f(\varepsilon(u))\dd x,
\qquad
\varepsilon(u):=\frac12(Du+Du^T),$$
where $\Omega\subset\mathbb R^3$ is a bounded Lipschitz domain
and $f\colon\Symthree\to\mathbb R$ is a continuous
energy density. Dependence on the symmetric gradient reflects
invariance under infinitesimal rigid motions. Under standard
$p$-growth and coercivity assumptions, symmetric quasiconvexity of $f$
provides the sequential weak lower semicontinuity of $\mathbb E$ (cf.~\cite{morrey1952quasi}) needed for
the direct method, while Korn's inequality yields coercivity
on the admissible set
\cite{barroso2000relaxation,fonseca1999cala,ebobisse2000lower,rindler2011lower}.

Writing $\pi_S(A):=(A+A^T)/2$, symmetric quasiconvexity (SQC) of $f$
is equivalent to ordinary quasiconvexity of $f\circ\pi_S$.
The same construction (see \cite{Kreisbeck19}) defines symmetric polyconvexity (SPC) and
symmetric rank-one convexity (SRC), giving the implications
\begin{align*}
    &f \text{ is convex } \Rightarrow f \text{ is SPC }
\Rightarrow      f \text{ is SQC } \Rightarrow f \text{ is SRC.}
\end{align*}
Related notions arise by restricting ordinary polyconvex
functions \cite{Ball1976} to symmetric matrices (PCS), testing quasiconvexity only
with scalar Hessians (HQC) (as introduced in \cite{meyers1965quasi}, see also \cite{fusco1980quasiconvexity,guidorzi1999lower}), or requiring convexity only along
directions $a\otimes a$ (RCS), giving the analogous chain of implications
\begin{align*}
    &f \text{ is convex } \Rightarrow f \text{ is PCS }
\Rightarrow      f \text{ is HQC } \Rightarrow f \text{ is RCS.}
\end{align*}
The corresponding properties satisfy (see Fig.~\ref{Fig 1: Diagram Semi-conveixty})
$$\mathrm{SPC}\Longrightarrow\mathrm{PCS},
\qquad
\mathrm{SQC}\Longrightarrow\mathrm{HQC},
\qquad
\mathrm{SRC}\Longrightarrow\mathrm{RCS}.$$
Understanding which of these implications can be reversed is a natural problem
in the calculus of variations. Although quasiconvexity and
polyconvexity are known to be distinct in general, the question whether
rank-one convexity implies quasiconvexity, posed by Morrey,
is considerably more difficult to answer. \v{S}ver\'ak
\cite{vsverak1992rank} constructed the first finite-valued counterexample
in $\mathbb R^{3\times2}$, Grabovsky
\cite{grabovsky2018microstructure} later obtained a different
construction in $\mathbb R^{8\times2}$, and Agazzi-Bruno-Guerra--Pasqualotto \cite{agazzi2026morreysproblemmathbbr2times} recently identified a connection between
these examples.
Also recently, Cassese
\cite{cassese2026solutionmorreysproblemmathbbr2times}
constructed a homogeneous quartic counterexample in
$\mathbb R^{2\times3}$, leaving $\mathbb R^{2\times2}$
as the only unresolved nontrivial matrix dimension
for the classical Morrey problem.

These questions are closely connected with the characterization
of generalized convexity notions (cf., e.g., \cite{Kreisbeck19,kristensen1999non,conti2003polyconvexity,grabovsky2019rank,harris2018two,muller1999rank,muller2003convex,martin2017rank,ball1981null,faraco2003quasiconvex,faraco2008tartar,tartar1979compensated})
the computation of relaxed energies (cf., e.g., \cite{GamKre2026,DolGam2026,KoStrII,allaire1993explicit,Kohn91,Firoozye93,dolzmann2000two,conti2013relaxation,conti2016relaxation,grabovsky1994bounds,grabovsky1995nonsmooth,antimonov2016phase,fechte2016lamination}),
and the geometry of microstructures (cf., e.g., \cite{ball1989fine,chenchiah2008relaxation,dolzmann1995microstructures,grabovsky2018microstructure,sebestyen2015laminates,grabovsky1995microstructures}). 
Further insight has come from recent numerical investigations
\cite{guerra2021numerical}, results for homogeneous isotropic
integrands \cite{KristensenGuerra2022,KristensenKirchheim2016}, and approaches involving
Burkholder integrands and heat-flow monotonicity
\cite{astala2012burkholder,
bruno2026quasiconvexitydacorognamarcellinienergy}.

Restricting these interrelations to symmetric matrices provides insights into the general case (see outlook) and is directly relevant to elasticity theory, specifically to geometrically linear elasticity and, via frame-indifferent lifting (cf.~\cite{ConDolDon2026}), also to nonlinear elasticity. We review the already established relations in
Section~\ref{Sec: Symmetric notions} and construct counterexamples
that resolve the remaining pairwise implications among the
seven properties considered here.

\subsubsection*{Main results}
Our first result separates symmetric quasiconvexity from polyconvexity at symmetric matrices.
\begin{Theorem}\label{Thm: SQC no PCS} Let $n\ge 2$. There exists a function $f_{1}\colon \Sym\to \R$ that is symmetric quasiconvex but fails to be polyconvex at symmetric matrices. 
\end{Theorem}
Notably, we also show that the quadratic form from \cite[Thm.~5.7]{Kreisbeck19}, which is symmetric quasi- but not symmetric polyconvex, is polyconvex at symmetric matrices. Thus the Boussaid--Kreisbeck--Schlömerkemper quadratic separates functions that are symmetric quasiconvex and polyconvex at symmetric matrices from those that are symmetric polyconvex, whereas our example separates symmetric quasiconvex functions from functions that are polyconvex at symmetric matrices.

Additionally, we show abstractly that the counterexamples in \cite[Cor.~5.4]{cassese2026solutionmorreysproblemmathbbr2times}, \cite{ConDolDon2026}, and \cite[Thm.~1.5]{agazzi2026morreysproblemmathbbr2times} can be modified into functions that are symmetric rank-one convex, but not Hessian quasiconvex.
\begin{corollary}\label{Cor: No Hessian Q Convex}
 There exists a $C^2$-function $f_{2}\colon \Symthree\to \R$ that is symmetric rank-one convex but fails to be Hessian quasiconvex at the origin. 
\end{corollary}
More generally, we prove the following scheme using techniques from \cite{kristensen1999non}.
\begin{Theorem}\label{Thm: Scheme for transfering RCS to SRC}
    Let $n\ge 2$, let $f\in C^2(\Sym,\R)$ be rank-one convex at symmetric matrices. Suppose there is an $A_0\in \Sym$ and a $\psi\in C^\infty(\mathbb T^n,\R)$ such that 
    $$\int_{\mathbb T^n} f(A_0+D^2 \psi)-f(A_0)\dd x<0. $$
    Then there exists a symmetric rank-one convex $\hat{f}\in C^2(\Sym,\R)$ having bounded Hessian such that $\hat{f}(A_0)=f(A_0)$ and
   $$\int_{\mathbb T^n} \hat{f}(A_0+D^2 \psi)-\hat{f}(A_0)\dd x<0. $$
\end{Theorem}

We also explain that symmetric rank-one convex polynomials of degree at most $3$ are already symmetric quasiconvex, and then construct a function from a quartic seed where this implication fails. We finally extend this polynomial to a function on the whole space, using the heat-flow argument from \cite{bruno2026quasiconvexitydacorognamarcellinienergy} to ensure Hessian quasiconvexity, yielding our last theorem. 
\begin{Theorem}\label{Thm: No SQC}
    There exists a smooth function $f_{3}\colon \Symthree\to \R$ with bounded Hessian that is symmetric rank-one convex and Hessian quasiconvex but fails to be symmetric quasiconvex at the origin. 
\end{Theorem}
As a note, another counterexample to the implication SRC $\implies $ SQC has been constructed independently by Conti, Dolzmann and Dondl \cite{ConDolDon2026}, in order to obtain a frame-indifferent counterexample to the ordinary Morrey conjecture in $\R^{3\times 3}$. Our construction is different in the polynomial seed, the localization and extension procedure, and the additional Hessian quasiconvexity. 

In addition, we give a detailed overview of all known implications and non-implications between the seven notions of convexity (cf.~Fig.~\ref{Fig 1: Diagram Semi-conveixty}). This paper, together with known results, settles the general pairwise implications between these.

\subsubsection*{Outline}

To avoid terminological confusion, we recall the necessary symmetric notions in the next Section \ref{Sec: Symmetric notions}. We also present the previously established interrelations between these notions, i.e., where implications hold, and where implications fail. In Section \ref{Sec: SQC does not imply PCS}, we prove Thm.~\ref{Thm: SQC no PCS}. 
In Section \ref{Sec: A scheme for SRC (Hesian QC)}, we establish our transfer scheme.
In Section \ref{Sec: On the strong symmetric Morrey problem} we discuss Morrey's problem in $\Symthree$ explicitly and present a counterexample based on a quartic polynomial, and finally, in the outlook Section, we discuss the potential impact of these findings.

\subsubsection*{Notation}
We write $\Symthree$ for the space of real symmetric
$3\times 3$ matrices, denote such matrices by letters $A,B,E,G,H$, and set
$$a\odot b:=\frac12(a\otimes b+b\otimes a).
$$
The Frobenius inner product and norm are denoted by
$A:B=\operatorname{tr}(A^TB)$ and $\vert A\vert=(A:A)^{1/2}$.
We write $\mathbb I$ for the identity matrix, indicating dimensions other than 3 by an index, and write
$\varepsilon(u):=\frac{1}{2}(Du+(Du)^T)$ for the symmetric gradient.
For scalar functions $\psi$, $D^2\psi$ denotes the Hessian,
for integrands $f$, $D^2f(A)[B,B]$ denotes the second
directional derivative at $A$ in direction $B$.
We use $\mathbb T^n=\mathbb R^n/\mathbb Z^n$, which we identify with $(0,1)^n$, equipped with the
Lebesgue measure. We write $C^\infty(\mathbb T,\R^d)$ for the space of $\mathbb Z^n$ periodic smooth function $v\colon \R^n\to \R^d$.
The notation $f^{\mathrm{sqc}}$ denotes the
symmetric quasiconvex envelope of $f$.

\subsubsection*{Acknowledgments}
The research for this paper has been conducted mostly during a research stay of the author at the University of Oxford. The author thanks Jan Kristensen, Gabriele Cassese, Georg Dolzmann, and Carolin Kreisbeck for many insightful conversations.

This work was funded by the Deutsche Forschungsgemeinschaft (DFG). We gratefully acknowledge the support within the project SPP 2256
(``Variational Methods for Predicting Complex Phenomena in Engineering Structures
and Materials''), project ID 441468770.

\subsubsection*{AI disclosure}

All of the three main results and their proofs have been conducted by the author. 
The author used ChatGPT-6 Astra
to find and fine tune the parameters for the quartic in Lemma.~\ref{Lem: quartic Pati polynomial}, and to investigate \v Sver\'ak's example $f_+$ and obtain the polyconvex measure where polyconvexity of $f_+$ fails (i.e., Lemma \ref{lem: HQC no PCS}). 
 Other than that, AI was only used to review a previous draft of this paper. The AI model is not an author of this work.

This paper has been written completely and checked
entirely by the author, who takes full responsibility for its contents.

\section{Symmetric notions of semi-convexity}
\label{Sec: Symmetric notions}

We recall the definitions and fix the terminology used below,
distinguishing the notions associated with symmetric gradients
from those associated with scalar Hessians.
\begin{Definition}
    A continuous function $F\colon \R^{3\times 3}\to \R$ is called  
    \begin{itemize}
        \item[(i)] polyconvex, if there exists a convex function 
    $$\hat{F}\colon \R^{3\times 3} \times \R^{3\times 3}\times\R\to \R  $$
    such that 
    $F (A)=\hat{F}(A,\cof A,\det A)$ for all $A\in \R^{3\times 3}$,
    \item[(ii)] quasiconvex, if for all $v\in C^\infty(\mathbb T^3,\R^3)$ and all $A\in \R^{3\times 3}$ there holds
    $$F(A)\le \int_{\mathbb T^3}F(A+D  v)\dd x ,$$
    \item[(iii)] rank-one convex, if for all $A,B\in \R^{3\times 3}$ with $\mathrm{rk}(A-B)=1$, and all $t\in (0,1)$, there holds
    $$F(A+t(B-A))\le (1-t)F(A)+ t F(B).$$
    \end{itemize}
\end{Definition}
Requiring $F=f\circ \pi_S$ to be quasiconvex leads to the following definitions.
\begin{Definition}
    A continuous function $f\colon \Symthree\to \R$ is called \emph{symmetric polyconvex} (SPC) (resp.~\emph{symmetric quasiconvex (SQC), symmetric rank-one convex (SRC), symmetric convex}), if the function $F=f\circ \pi_S$ is polyconvex (resp.~quasiconvex, rank-one convex, convex). 
\end{Definition}
Writing out these conditions leads to the following simple characterization.
\begin{lemma}[cf.{\cite[Prop.~2.2]{Kreisbeck19}}]
    A continuous function $f\colon \Symthree \to \R$ is 
    \begin{itemize}
        \item[(i)] \emph{symmetric} convex, if and only if $f$ is convex,
        \item[(ii)] \emph{symmetric} polyconvex, if and only if there is a jointly convex function $\hat{f}\colon \Symthree\times \Symthree\to \R$, non-increasing with respect to the second order minors (in the Loewner sense), such that $f(A)=\hat{f}(A,\cof A)$ for all $A \in \Symthree$
        (\cite[Thm.~4.1 \& Thm.~5.1]{Kreisbeck19}, \cite{Merabet24}),
        \item[(iii)] \emph{symmetric} quasiconvex, if and only if for all $A\in \Symthree$ there holds
        $$f(A)\le \inf_{v\in C^{\infty}(\mathbb T^3,\R^3)}\int_{\mathbb T^3}f(A+\varepsilon(v))\dd x ,$$
        \item[(iv)] \emph{symmetric} rank-one convex, if and only if for all $t\in [0,1]$, $A,B\in \Symthree$, such that $A-B=a\odot b$ for some $a,b\in \R^3$, there holds
        $$f(tA+(1-t)B)\le tf(A)+(1-t)f(B).$$
    \end{itemize}
\end{lemma}

Testing with Hessians instead of symmetrized gradients leads to the following different notions.
\begin{itemize}
    \item[(i)]  \emph{Polyconvex at symmetric matrices (PCS)}. This means that there is an ordinary polyconvex function $F\colon \R^{3 \times 3}\to \R$, such that $f=F_{\vert_{\Symthree}}$. An example is the determinant function, which is not symmetric polyconvex. 
    \item[(ii)] \emph{Hessian quasiconvexity (HQC)}, also called $2$-quasiconvexity, meaning that for all $\psi\in C^\infty(\mathbb T^3,\R)$ and for all $A\in \Symthree$, there holds
    $$f(A)\le\int_{\mathbb T^3} f(A+D^2 \psi)\dd x.$$
    This is a weaker notion than symmetric quasiconvexity, as the determinant again shows. 
    For smooth strictly Hessian quasiconvex integrands satisfying $(p>1)$-growth and a Lipschitz condition, it is well-known that they can be extended to ordinary quasiconvex functions on the whole space (see \cite{dal2004higher} and for earlier work \cite{muller2003convex}).
    \item[(iii)]
    \emph{Rank-one convex at symmetric matrices (RCS)}, i.e., for all $A\in \Symthree$ and for all $a\in \R^3$, the function $\R\to \R$, $f(A+ta\otimes a)$ is convex in $t$. Crucially, these are precisely the lines of ordinary rank one among the symmetrized rank-one lines $a\odot b$. Consequently, being rank-one convex at symmetric matrices is a weaker property than 
symmetric rank-one convexity, strictly, in fact, as the determinant again shows.
\end{itemize}

We essentially have seven different notions of symmetric semi-convexity (including regular convexity). We now give an overview, which notion implies which other notions, and where reverse implications fail. 

The left chain of implications in Fig.~\ref{Fig 1: Diagram Semi-conveixty} is immediate from the ordinary counterpart, a proof of which can be found in \cite[Ch.~3,5]{Daco08}. 
The right chain of implications is well-known: the first implication is trivial, the second, polyconvexity at symmetric matrices implies Hessian quasiconvexity, follows from \cite[Thm.~5.3]{ball1981null}, and the third, Hessian quasiconvexity implies rank-one convexity at symmetric matrices, from \cite[Thm.~3.3]{ball1981null}.

The implications from the left to the right have been assessed. Crucially, all reversals of these implications fail. For instance, $-\tr(\cof(\cdot))$ is an SPC function that is not convex (see \cite{Kreisbeck19}). The determinant is a polyconvex (and therefore HQC and RCS) function that is neither convex, SPC, SQC, nor SRC. 

An SQC function that is not SPC was presented in \cite[Thm.~5.7]{Kreisbeck19}, and a function that is rank-one convex on symmetric matrices but not HQC was recently constructed in \cite[Cor.~5.4]{cassese2026solutionmorreysproblemmathbbr2times} and \cite[Thm.~1.5]{agazzi2026morreysproblemmathbbr2times}.

\begin{figure}
\begin{center}
\begin{tikzcd}[
    row sep=2cm, 
    column sep=5.5cm, 
    ampersand replacement=\&,s
    every arrow/.append style={thick, {Stealth[scale=1.5]}}
]
\text{Convex} 
    \arrow[r,leftrightarrow] 
    \arrow[d, shift left=1.5] 
    \& 
\text{Convex at Sym} 
    \arrow[d, shift right=1.5] 
\\
\text{SPC} 
    \arrow[u, dotted, shift left=1.5, "-\text{tr\,cof}"] 
    \arrow[r, shift left=0.0 ] 
    \arrow[d, shift left=1.5] 
    \& 
\text{PCS} 
    \arrow[u, dotted, shift right=1.5, "\text{det}"'] 
    \arrow[l, dotted, shift right=2, "\text{det}"' ] 
    \arrow[d, shift right=1.5] 
    \arrow[ld, dotted, shift right=2.0, "\text{Thm.~\ref{Thm: SQC no PCS}}" ] 
\\
\text{SQC} 
    \arrow[u, dotted, shift left=1.5, "\text{\cite[Thm.~5.7]{Kreisbeck19}}"] 
    \arrow[r, shift left=2 ] 
    \arrow[d, shift left=1.5] 
    \arrow[ur, dotted, shift left=0.0, "\text{det}" ] 
    \& 
\text{HQC} 
    \arrow[u, dotted, shift right=1.5, "\text{\cite{vsverak1992new}}"'] 
    \arrow[l, dotted, shift right=0.0, "\text{det}" ] 
    \arrow[d, shift right=1.5] 
    \arrow[ld, dotted, shift right=2, "\text{Cor.~\ref{Cor: No Hessian Q Convex}}" ]
\\
\text{SRC} 
    \arrow[u, dotted, shift left=1.5, "\text{Thm.~\ref{Thm: No SQC}}"] 
    \arrow[r, shift left=0.0 ] 
    \arrow[ur, dotted, shift left=0.0, "\text{det}" ] 
    \& 
\text{RCS} 
    \arrow[u, dotted, shift right=1.5, "\text{\cite[Cor.~5.4]{cassese2026solutionmorreysproblemmathbbr2times}}"'] 
    \arrow[l, dotted, shift left=2.0, "\text{det}" ]
\end{tikzcd}
\end{center}
\caption{Relation between symmetric notions of semi-convexity. Arrows indicate an implication, dotted arrows indicate failure of the associated implication, with the corresponding label giving a counterexample.}
    \label{Fig 1: Diagram Semi-conveixty}
\end{figure}
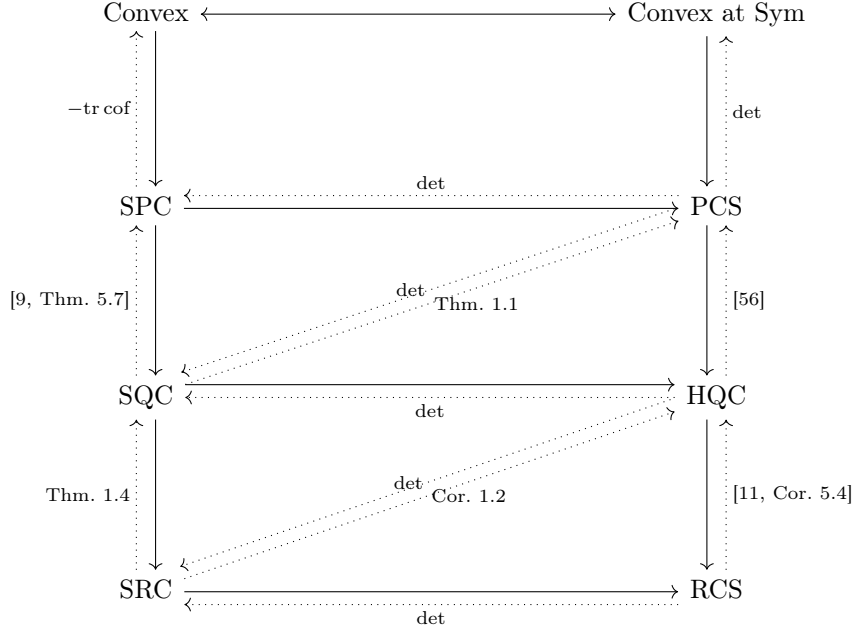

\begin{Bem}\begin{itemize}
    \item[a)] The requirement that $f$ be continuous can be relaxed to $f$ being finite-valued, Borel measurable, and locally bounded.
    \item[b)] In dimension five, extending the construction in \cite{vsverak1992rank}, there is a simple counterexample for the implication SRC $\implies$ SQC by considering the three-dimensional subspace
$$\begin{pmatrix}
    0 & 0 &  x & 0 & z \\
    0& 0  & 0 & y & z\\
    x & 0 & 0 & 0 & 0 \\
    0 & y & 0 & 0 & 0 \\
    z & z & 0 & 0 & 0 
\end{pmatrix}.$$
\end{itemize}
\end{Bem}

\section{SQC does not imply PCS}
\label{Sec: SQC does not imply PCS}

We now investigate whether symmetric quasiconvexity implies
polyconvexity at symmetric matrices. First, we show that the
quadratic example of \cite[Thm.~5.7]{Kreisbeck19} is PCS,
and therefore does not disprove this implication.
The example reads
$$f_\eta(A):=(A_{12}-A_{13})^2+(A_{13}-A_{23})^2+(A_{23}-A_{12})^2+A_{11}^2+A_{22}^2+A_{33}^2-\eta \vert A\vert^2,$$
where 
$$\eta:=\min_{a,b\in \R^3, \vert a\odot b\vert=1}f_0(a\odot b).$$
\begin{Proposition}
    The function $f_\eta$ is polyconvex at symmetric matrices. 
\end{Proposition}
\begin{proof}
    We first establish an upper bound on $\eta$. This can easily be done by testing with the matrix
    $$A_1:=\frac{1}{\sqrt{8}}\begin{pmatrix}
        1 & 1 & 1\\ 
        1 & 1 &1 \\
        1 & 1 &0
    \end{pmatrix}=\frac{1}{\sqrt{8}}(1,1,2)\odot (1,1,0).$$
    We see that $f_0(A_1)=\frac{1}{4}$, hence $\eta\le \frac{1}{4}$. 
Now, let us use the (full-rank) matrix
$$A_0=\begin{pmatrix}
        0 & 1 & 1\\ 
        1 & 0 &1 \\
        1 & 1 &0
    \end{pmatrix}.$$
    Let us also write $d=(A_{11},A_{22},A_{33})$ for the diagonal and $o=(A_{23},A_{13},A_{12})$ for the off-diagonal part of a matrix $A$. 
Then
$$A_0:\cof A=2(o_1o_2+o_1o_3+o_2o_3)-2d:o,$$
and
$$f_\eta(A)=(1-\eta)\vert d\vert^2+2(1-\eta)\vert o\vert^2-2(o_1o_2+o_1o_3+o_2o_3).$$
We can split up $f_\eta$ using
$$h(A):=f_\eta(A)+A_0:\cof(A)=(1-\eta)\vert d\vert^2+2(1-\eta)\vert o\vert^2-2d:o.$$
Setting $t=1-\eta$ and completing the square gives
$$ h(A)=t\biggl\vert d-\frac{1}{t} o\biggl\vert^2+\biggl(2t-\frac{1}{t}\biggr)\vert o\vert^2. $$
By our bound on $\eta$, we infer
$$t\ge \frac{3}{4}, \quad \text{and} \quad 2t-\frac{1}{t}\ge \frac{3}{2}-\frac{4}{3}=\frac{1}{6}\ge 0.$$
Hence, the quadratic form $h$ is nonnegative and therefore convex. Thus, $f_\eta$ is PCS as a sum of $-A_0:\cof$ and a convex function.
\end{proof}
Let us now continue with our discussion of an SQC function that fails to be polyconvex. 
Consider \v{S}ver\'ak's function 
\begin{equation}
    f_+(\varepsilon):=\begin{cases}
        \det \varepsilon &\text{if } \varepsilon \text{ is positive definite,}\\
        0 &\text{otherwise,}
    \end{cases} \label{eq: Def f+}
\end{equation}
treated in \cite{vsverak1992new}. It is not SQC (actually not SRC, take e.g.~$
f_+\bigl(\diag(1+t,1-t,1)\bigr)=1-t^2$), but
it is an HQC function that is not  PCS. We present a short argument for this.
\begin{lemma}\label{lem: HQC no PCS}
    The function $f_+\colon \Symthree\to \R$ is Hessian quasiconvex but fails to be polyconvex at symmetric matrices. 
\end{lemma}
\begin{proof}
The Hessian quasiconvexity of this positive determinant was established in \v{S}ver\'ak \cite{vsverak1992new} (see also \cite{guerra2019extremal}). To see that it is not polyconvex, take
$$A_1=\diag(5,-1,1), \quad A_2=\diag(-1,5,1), \quad A_3=\diag(-7,-7,1),$$
and the measure (convex combination)
$$\mu:=\frac{4}{9}\delta_{A_1}+\frac{4}{9}\delta_{A_2}+\frac{1}{9}\delta_{A_3}.$$
Then
the barycenter of $\mu$ is $\mathbb I$, while
$$\int_{\Symthree}\det\dd\mu= 1=\det(\mathbb I), \quad \text{and} \quad \int_{\Symthree}\cof\dd\mu= \mathbb I=\cof(\mathbb I),$$
but
$$\int_{\Symthree} f_+\dd \mu= 0<1=f_+(\mathbb I).$$
\end{proof}

Although \v{S}ver\'ak's function is not SQC, it can be used to show that there is an SQC function that is not PCS, as we show now.

\begin{proof}[Proof of Thm.~\ref{Thm: SQC no PCS}]
    Let us define $A_1,A_2,A_3$ as in lemma \ref{lem: HQC no PCS}. Let $K=\{A_1,A_2,A_3\}$. Set 
    $$f_{1}(A):=\bigl(\mathrm{dist}^2(A,K)\bigr)^{\mathrm{sqc}}.$$
    Identifying $(0,1)^3$ with $\mathbb T^3$, and using the  periodic Dacorogna formula \cite[Prop.~5.13]{Daco08} (adapted to the function composed with the projection), we obtain
$$f_1(A)
=
\inf_{v\in C^\infty(\mathbb T^3,\mathbb R^3)}
\int_{(0,1)^3}\mathrm{dist}^2(A+\varepsilon(v),K)\dd x.$$
    Necessarily,
    $$0\le f_1(A_i)\le \mathrm{dist}^2(A_i,K)=0$$
    for all $i=1,2,3$, while the same measure $\mu$ violates the Jensen inequality for PCS once we can show that $f_1(\mathbb I_3)>0$. \\
    \textit{1.~2d reduction: }Let us write $\Pi\colon \Symthree\to \Symtwo$ for the projection onto the upper left corner. We set $B_i=\Pi A_i$ and $K'=\{B_1,B_2,B_3\}\subset \Symtwo$. 
    Suppose that there is a $\delta>0$ such that
    for any function $v\in C^\infty(\mathbb T^2,\mathbb R^2)$ there holds
$$\int_{(0,1)^2}\mathrm{dist}^2(\mathbb I_2+\varepsilon(v)(x'),K')\dd x'\ge \delta.$$
    Now, for $u\in C_{\mathrm{per}}^\infty((0,1)^3,\R^3)$, we introduce a family of functions
    $$v_{x_3}:=\bigl(u_1(\cdot,x_3),u_2(\cdot,x_3)\bigr)\in C^\infty(\mathbb T^2,\mathbb R^2).$$
Using 
    $$\mathrm{dist}(A,K)\ge \mathrm{dist}(\Pi A,K'), $$
we infer by Fubini's Theorem
\begin{align*}
    \int_{(0,1)^3} \mathrm{dist}^2(\mathbb I_3+\varepsilon(u)(x),K)\dd x &\ge \int_0^1\int_{(0,1)^2}\mathrm{dist}^2(\mathbb I_2+\varepsilon(v_{x_3})(x'),K')\dd x'\dd x_3\\ 
    &\ge \int_0^1 \delta \dd s=\delta.
\end{align*}
\textit{2.~Vanishing curl in the limit: }Let us assume that $$\bigl(\mathrm{dist}^2(\cdot,K')\bigr)^{\mathrm{sqc}}(\mathbb I_2)=0.$$ Then by the Dacorogna formula, there is a sequence $(v_n)_{n\in \N}\subset C^\infty(\mathbb T^2,\mathbb R^2)$ such that
$$\int_{(0,1)^2}\mathrm{dist}^2(\mathbb I_2+\varepsilon(v_n),K')\dd x \to 0.  $$
For every $n$, let $Z_n\colon (0,1)^2\to K'$ assign to every $x\in (0,1)^2$ the $B_i$ nearest to $\mathbb I_2+\varepsilon(v_n)(x)$ (if this is not unique, choose the $B_i$ with the smallest index). 
By construction, 
\begin{align}
    \Vert \mathbb I_2+\varepsilon(v_n)-Z_n\Vert_{L^2}^2=\int_{(0,1)^2}\mathrm{dist}^2(\mathbb I_2+\varepsilon(v_n),K')\dd x \to 0. \label{eq: L2 convergernce}
\end{align}
Since the $Z_n$ are uniformly bounded in $L^\infty((0,1)^2,\Symtwo)$, the sequence $E_n:=\mathbb I_2+\varepsilon(v_n)$ is bounded in $L^2((0,1)^2,\Symtwo)$. 

Let us now set 
$$\theta_{i,n}:=\vert \{x\in (0,1)^2\mid Z_n(x)=B_i\}\vert.$$
Then by \eqref{eq: L2 convergernce} and $\int_{(0,1)^2} E_n=\mathbb I_2$, we have
$$ \int_{(0,1)^2} Z_n\dd x=\sum_{i=1}^3\theta_{i,n}B_i\to \mathbb I_2  \quad \text{and} \quad \sum_{i=1}^3\theta_{i,n}=1.$$
Restricting to a convergent subsequence, we may assume that the $\theta_{i,n}$ converge to $\theta_i$, the unique solution of the limit of the system of equations 
$$\theta_1=\theta_2=\frac{4}{9}, \theta_3=\frac{1}{9}.$$
Crucially, by the choice of $A_i$, we also get
$$ \int_{(0,1)^2} \det Z_n\dd x=\sum_{i=1}^3\theta_{i,n}\det B_i\to \sum_{i=1}^3\theta_{i}\det B_i=1.$$
By \begin{equation}
    \vert \det A-\det B\vert\le C(\vert A\vert+ \vert B\vert)\vert A-B\vert, \label{eq: Det inequality}
\end{equation}(for some $C>0$),
Hölder's inequality, and
\eqref{eq: L2 convergernce} we deduce
\begin{equation}
\int_{(0,1)^2} \det E_n\dd x\to 1. \label{eq: Det limit}
\end{equation}
Next, we use the operator $\operatorname{curl}v:=\partial_1 v_2-\partial_2 v_1$, which captures the skew-symmetric part of the gradient. Expanding the determinant gives
\begin{align*}
    \int_{(0,1)^2} \det (\mathbb I_2+\varepsilon(v))\dd x&=  \int_{(0,1)^2} 1+\operatorname{div}v+\det Dv-\frac{1}{4}\vert \operatorname{curl}v\vert^2\dd x\\
    &=1-\frac{1}{4}\int_{(0,1)^2}\vert \operatorname{curl}v\vert^2\dd x.
\end{align*}
Thus, for the sequence $(v_n)_{n\in \N}$, using \eqref{eq: Det limit}, we infer
\begin{equation}
    \int_{(0,1)^2}\vert \operatorname{curl}v_n\vert^2\dd x\to 0. \label{eq: curl free limit}
\end{equation}
For each $n\in \N$, there is a Helmholtz decomposition (cf., e.g.~\cite[Lem.~1]{dal2004higher}) of the $\mathbb Z^2$-periodic smooth functions
$$v_n=D  \psi_n+w_n+c_n,$$
where $\psi_n\in C_{\mathrm{per}}^\infty(\mathbb T^2,\R)$, $c_n\in \R^2$, and $w_n\in C_{\mathrm{per}}^\infty(\mathbb T^2,\R^2)$, admitting
$$\int_{(0,1)^2}\psi_n\dd x=0, \quad \operatorname{div}w_n=0, \quad \text{and} \quad \int_{(0,1)^2}w_n\dd x=0. $$
Then, we deduce
\begin{align*}
    \int_{(0,1)^2}\vert Dw_n\vert^2\dd x&= \int_{(0,1)^2}\vert \operatorname{div}w_n\vert^2+\vert \operatorname{curl}w_n\vert^2\dd x=\int_{(0,1)^2}\vert \operatorname{curl}v_n\vert^2\dd x.
\end{align*}Hence, together with \eqref{eq: curl free limit} and \eqref{eq: L2 convergernce}, we conclude
\begin{equation}
    \Vert \mathbb I_2+D^2\psi_n-Z_n\Vert_{L^2}\to 0. \label{eq: Hessian approximation}
\end{equation}
\textit{3.~Comparison with $f_+$: }Lastly, consider the HQC function $f_+$ (for $\Symtwo$ matrices) in \eqref{eq: Def f+}. We have for all $n\in \N$
\begin{equation}
    \int_{(0,1)^2}f_+(\mathbb I_2+D^2\psi_n)\dd x\ge f_+(\mathbb I_2)=1. \label{eq: f+ HQC}
\end{equation} 
This positive determinant admits the same estimate as in \eqref{eq: Det inequality},
\begin{equation*}
    \vert f_+( A)-f_+(B)\vert\le C(\vert A\vert+ \vert B\vert)\vert A-B\vert.
\end{equation*}
If both matrices are positive definite, this is the determinant estimate. If neither is positive definite, the left-hand side vanishes. If precisely one is positive
definite, the segment joining them meets the boundary of the positive-definite
cone at some matrix $M$ with $\det M = 0$.
 
Note that $G_n:=\mathbb I_2+D^2\psi_n$ is also bounded in $L^2$, whence by Hölder's inequality and $L^2$-convergence \eqref{eq: Hessian approximation}, we see that
\begin{align*}
     \int_{(0,1)^2}\vert f_+( G_n)-f_+(Z_n)\vert \dd x&\le C(\Vert G_n\Vert_{L^2}+\Vert Z_n\Vert_{L^2})\Vert G_n-Z_n\Vert_{L^2}\to 0.
\end{align*}
But this is a contradiction to \eqref{eq: f+ HQC}, since $f_+(Z_n)\equiv 0$. Consequently, the two-dimensional infimum, and thus the three-dimensional infimum, are strictly positive, contradicting the Jensen inequality of PCS. Finally, the case $n\ge 4$ works analogously by extending the tensors $A_i$ with ones on the diagonal. 
\end{proof}

\section{Transferring RCS to SRC without HQC}
\label{Sec: A scheme for SRC (Hesian QC)}

Starting from the examples in \cite[Cor.~5.4]{cassese2026solutionmorreysproblemmathbbr2times} and \cite[Thm.~1.5]{agazzi2026morreysproblemmathbbr2times}, which are rank-one convex functions at symmetric matrices without being HQC, we show that these functions can be perturbed in a way that yields SRC functions that are not HQC at the origin. 

The cited examples are RCS but not SRC. Cassese's integrand satisfies
$$h(tB)=(4\delta-1)t^4,
\quad
B=e_1\odot(2e_2),$$
and therefore fails symmetric rank-one convexity for\footnote{Increasing $\delta$ cannot restore symmetric rank-one convexity
while retaining the failure of Hessian quasiconvexity at the origin.
Indeed, for $\delta\ge 1/4$, the estimate $q(A)\ge -\vert A\vert^4/4$
implies $h\ge 0=h(0)$.}
$0<\delta<1/4$.\\
The integrand of Agazzi--Bruno--Guerra--Pasqualotto admits (in their notation)
$$A=X(0,0,-3,1),\quad
B=\operatorname{diag}(1,-1,0)
,\quad F_{\alpha,\lambda,\mu}(A+sB)
=-3s^2+\alpha(12+2s^2)^2,$$
hence being SRC would require $\alpha\ge 1/16$.
However, the inequality
$r_1r_2r_3r_4\leq \vert X(r)\vert^4/16$
then implies $F_{\alpha,\lambda,\mu}\ge 0
=F_{\alpha,\lambda,\mu}(0)$, and hence HQC
at the origin.

Our argument uses an added quadratic Hessian null Lagrangian to enforce
convexity along symmetrized rank-one directions without interfering with
the Hessian quasiconvexity deficit. Our localization construction is based on the one introduced by Kristensen \cite[Prop.~2]{kristensen1999non}.

\begin{proof}[Proof of Thm.~\ref{Thm: Scheme for transfering RCS to SRC}]
Replacing $f(A)$ by $f(A_0+A)-f(A_0)$, we may assume
that $A_0=0$ and $f(0)=0$. Once the construction is complete,
translation of the matrix variable gives the conclusion
at the original base point. We first give the proof in three dimensions.  
We proceed in four steps.\\
    \textit{1.~Preliminaries: }By assumption, there are constants 
    $$-\int_{\mathbb T^3} f(D^2 \psi)\dd x=C_1>0 \quad \text{and} \quad \int_{\mathbb T^3} \vert D^2 \psi\vert^2\dd x=C_2>0.$$
Choose $\epsilon>0$ such that 
\begin{align*}
    \epsilon C_2<\frac{C_1}{2}. 
\end{align*}
Let us define the perturbed functions
$$f_\epsilon(A):=f(A)+\epsilon \vert A\vert^2.$$
Then 
\begin{align}
    \int_{\mathbb T^3} f_\epsilon(D^2 \psi)\dd x < -\frac{C_1}{2}, \label{eq: f eps not HQC}
\end{align}while for every $B =a\otimes a$, $a\in \R^3$,
\begin{align*}
    D^2 f_\epsilon (A)[B,B]\ge 2 \epsilon \vert B \vert^2.
\end{align*}
Let us now introduce the quadratic function $q(A)=-2\tr\cof(A)$. This is an SPC (thus SQC and SRC) function which is (ordinary) polyaffine, and thus a Hessian null Lagrangian. Restricted to symmetric rank-one tensors, it vanishes precisely on those of the form $\pm a\otimes a$, since the cofactor matrix of such tensors is zero, and negative semi-definite in general, with a nonzero eigenvalue if the matrix does not have rank one.
Due to its quadratic nature, we have 
\begin{align*}
    D^2 q(A)[B,B]=2q(B)
\end{align*}
for all $A,B\in \Symthree$. \\
\textit{2.~SRC on a bounded set: }We now combine $f_\epsilon$ and $q$ to get the family of functions
$$g_\lambda=f_\epsilon+\lambda q,$$
where $\lambda\ge 0$. Let us set $R:=\sup_{x\in \mathbb T^3}\vert D^2\psi(x)\vert $. We claim that there is a $\lambda\ge 0$ such that $g_\lambda$ is SRC on $B(0,2R)\subset \Symthree.$
Let us define two closed subsets of the unit sphere 
\begin{align*}
    \mathbb X&:=\{B  \in \Symthree\mid \vert B \vert=1, \text{ and }B =a\odot b\text{ for some }a,b\in \R^3\},\\
    \mathbb X_0&:=\{B  \in \Symthree\mid \vert B \vert=1, \text{ and }B =\pm a\otimes a\text{ for some }a\in \R^3\}.
\end{align*}
Since for all $A\in\overline{B_{2R}}$ and $B \in\mathbb X_0$, there holds 
$$D^2f_\epsilon(A)[B ,B ]\ge 2\epsilon,$$
compactness and continuity yield a relatively open neighborhood
$\mathbb U$ of $\mathbb X_0$ in $\mathbb X$ such that for all $B\in \mathbb U$
$$D^2f_\epsilon(A)[B ,B ]\ge \epsilon.$$
Because $q\ge 0$ on $\mathbb X$, the same bound holds for
$g_\lambda$ for every $\lambda\ge 0$.

If $\mathbb U$ can be chosen to equal $\mathbb X$, then take $\lambda=0$. If $\mathbb X\setminus \mathbb U$ is nonempty, this set is closed by definition and bounded, hence the minimum 
$$C_3:=\min_{B  \in \mathbb X\setminus \mathbb U}q(B )>0.$$
is attained. 
Setting 
$$C_4:=\max\biggl\{0,-\min_{B  \in \mathbb X,A\in \overline{B(0,2R)}}D^2 f_\epsilon(A)[B ,B ]\biggr\},$$
which exists due to compactness, we can set 
\begin{align}
    \lambda :=\frac{C_4+\epsilon}{2M}. \label{eq: Def lambda}
\end{align}
For every $A\in \overline{B(0,2R)}$ and $B \in \mathbb X\setminus \mathbb U$, there holds
$$D^2 g_\lambda (A)[B ,B ]\ge D^2 f_\epsilon (A)[B ,B ]+\lambda 2q(B)\ge -C_4+2\lambda C_3\ge \epsilon.$$
From homogeneity, we infer the claim for all $H \in \Symthree$ that are of symmetric rank-one.\\
\textit{3.~A convex correction: }Let us choose a function $\zeta\in C^\infty([0,\infty),[0,1])$ such that $\zeta \equiv0$ on $[0,R]$ and $\zeta\equiv 1$ on $(\tfrac{3}{2}R,\infty)$. Let $\rho\colon [0,\infty)\to \R$ be its second antiderivative satisfying $\rho(0)=0$ and $\rho'(0)=0$. 
We define a convex function
$$h\colon \Symthree\to \R, \quad A\mapsto \rho(\vert A\vert).$$
We claim that $h$ is strongly convex in $\Symthree\setminus \overline{B(0,2R)}$. First, its derivative, writing $r:=\vert A\vert$, is given by
$$ Dh(A)[B]=\rho'(r)\frac{A:B}{r},$$
hence its second derivative is given by
\begin{align*}
    D^2h(A)[B,B]=\zeta(r)\frac{(A:B)^2}{r^2} + \frac{\rho'(r)}{r^3}\bigl(\vert B\vert^2r^2-(A:B)^2\bigl ).
\end{align*}
For $r\ge 2R$, it holds that
$$\frac{\rho'(r)}{r}=\tfrac{1}{r}\int_0^r\zeta (s)\dd s\ge \tfrac{1}{r}\int_{\tfrac{3R}{2}}^{r}\dd s=\frac{r-\tfrac{3R}{2}}{r}\ge \frac{1}{4}.$$
Moreover, 
$$\tfrac{1}{r}\int_0^r\zeta (s)\dd s\le  \tfrac{1}{r}\int_{R}^{r}\dd s=\frac{r-R}{r}\le 1,$$
and consequently
$$ D^2h(A)[B,B]\ge \frac{1}{4}\vert B\vert^2. \label{eq: h strongly convex}$$
\textit{4.~Global extension and failure of  HQC: }Let us now choose a smooth cutoff function $\chi\in C^\infty(\Symthree,[0,1])$ such that $\chi\equiv 1$ on $\overline{B(0,2R)}$ and $\chi\equiv 0$ on $\Symthree \setminus B(0,3R)$. Since $\chi f_\epsilon$ is $C^2$ and compactly supported, there is a constant $C_0\ge 0$ such that
$$D^2 (\chi f_\epsilon)(A)[B,B]\ge -C_0\vert B\vert^2$$
for every $A,B\in \Symthree$. 
Take $\nu:=4(C_0+\epsilon)$, 
and define
\begin{align*}
    \hat{f}=\chi f_\epsilon+\lambda q+\nu h.
\end{align*}
For every compatible direction $B$, we have
$$D^2\hat f(A)[B,B]\ge 
\begin{cases}
D^2g_\lambda(A)[B,B]
\ge \epsilon\vert B\vert^2,
& \vert A\vert\leq2R,\\
(-C_0+\nu/4)\vert B \vert^2
=\epsilon\vert B \vert^2,
& \vert A\vert\ge 2R.
\end{cases}$$
Here the first estimate uses $\chi=1$ and the convexity of $h$,
whereas the second uses $q(B)\ge 0$ and the lower bound
for $D^2h$.
It remains to show that $\hat{f}$ is not HQC at zero. By definition, $D^2\psi(x)\in \overline{B(0,R)}$ for all $x\in \mathbb T^3$. In particular, 
$$\chi(D^2\psi(x))=1, \quad \text{and} \quad h(D^2\psi(x))=0,$$
hence $\hat{f}(D^2\psi(x))=f_\epsilon(D^2\psi(x))+\lambda q(D^2\psi(x))$. Using the Hessian null Lagrangian property together with \eqref{eq: f eps not HQC}, we conclude
\begin{align*}
      \int_{\mathbb T^3} \hat{f}(D^2 \psi)-\hat{f}(0)\dd x &=   \int_{\mathbb T^3} f_\epsilon(D^2 \psi)+\lambda q(D^2 \psi)\dd x\\
      &= \int_{\mathbb T^3} f_\epsilon(D^2 \psi)\dd x< -\frac{C_1}{2}<0.
\end{align*}
As a final remark, we can extend the proof to arbitrary dimension by choosing the translator $q$ to be $-2$ times the trace of the second order minors matrix. 
\end{proof}

\section{SRC and HQC do not imply SQC}
\label{Sec: On the strong symmetric Morrey problem}

A \v Sverák-type subspace construction was first proposed by Pati (see \cite[Theorem~3]{pati2025calculus})
to construct a counterexample for the implication SRC $\implies$ SQC. The subspace $\mathbb L$ considered is defined via the map
    \begin{equation}
        S\colon \R^3\to \Symthree,\quad (x,y,z)\mapsto \begin{pmatrix}
        x & 0 & \tfrac{z}{2}\\
        0 & y & \tfrac{z}{2}\\
        \tfrac{z}{2}&\tfrac{z}{2} &0
    \end{pmatrix},\label{eq: Pati subspace}
    \end{equation}
through $\mathbb{L}:=S(\R^3)$. Note that we normalized $S$ to be an isometry w.r.t.~the Frobenius norm.
A matrix in $\mathbb{L}$ is compatible (i.e., having symmetrized rank-one) if and only if
    $$z=0, xy\le 0 \quad \text{or} \quad x+y=0. $$
    Indeed, $\det S(x,y,z)= -z^2\tfrac{x+y}{4}$ and $\tr  S(x,y,z)= x+y$, and for the symmetric rank-one property, either the matrix has ordinary rank one, or it has ordinary rank two and the two nonzero eigenvalues have opposite sign.

However, the characterization of compatible directions used
in \cite[Lemma~1]{pati2025calculus} omits rank-two symmetrized rank-one lines. 
The function $g$ discussed in the paper is not symmetric rank-one affine, but strictly concave along the symmetric rank-one line 
    $$\begin{pmatrix}
        0 & 0 & -1\\
        0&0 & -1\\
        -1&-1&0
    \end{pmatrix}+t\begin{pmatrix}
        1 & 0 & 0\\
        0&-1 & 0\\
        0&0&0
    \end{pmatrix}.$$
Keeping the subspace construction and the test function within this example, while replacing the cubic seed by a quartic, we can still construct an SRC integrand that is not SQC. 
Furthermore, this example can be made HQC, without enforcing SQC.

Before we present our solution, let us first explain briefly why a globally cubic polynomial cannot provide such a counterexample. The fact that every SRC quadratic polynomial is already SQC can be obtained directly from the non-symmetric counterpart established in \cite{van1947extension,van1949signe}, see also \cite{harutyunyan2015explicit,harutyunyan2017towards,zhang2002equality}. 
Let $p\colon\Symthree\to\mathbb R$ be a polynomial of degree
at most $3$, and write $p=p_3+p_2+p_1+p_0$,
where $p_i$ is homogeneous of degree $i$.
Let us assume that $p$ is SRC, meaning that for all $A\in \Symthree$, $B=a\odot b$, $a,b\in \R^3$, and $t\in \R$, there holds
$$0\le D^2p(tA)[B,B]=tD^2p_3(A)[B,B]+D^2p_2(A)[B,B]. $$
This can only hold if
$$D^2p_3(A)[B,B]=0, $$
in other words, if $p_3$ is symmetric rank-one affine. This, however, requires $p_3$ to be affine (cf.~\cite[Prop.~5.10]{Kreisbeck19}), and thus zero. 

Next, we show that there is a polynomial of degree four on the three-dimensional subspace considered by Pati, which gives rise to a SRC and HQC but not SQC function. 
\begin{lemma}\label{Lem: quartic Pati polynomial}
    There exists a quartic polynomial $p\colon \R^3\to \R$, such that 
    \begin{equation}\label{eq: polynomial is src}
        D^2p(X)[B,B]\ge 0
    \end{equation}
    holds for all $X\in \R^3$ and all $B=(a,b,c)$ admitting 
    $$c=0,\ ab\le 0 \quad \text{or} \quad a+b=0,$$
    and $p(0)=0$, while
    $$\int_{(0,2\pi)^3}p(\cos x,\cos y,\cos(x+y))\dd (x,y,z)<-\frac{3}{2}.$$
\end{lemma}
\begin{proof}
    We take the polynomial
    \begin{align}\label{eq: quartic polynomial example}
        \begin{aligned}
p(x,y,z)=&
\frac15(x^2+y^2)+\frac7{10}z^2-30xy
-30xyz-\frac53(x^2+y^2)z+\frac32z^3\\
&+\frac23(x^4+y^4)-2(x^3y+xy^3)+3x^2y^2
+4(x-y)^2z^2+\frac43z^4.
\end{aligned}
    \end{align}
    Using trigonometric orthogonality, we obtain
    $$\int_{(0,2\pi)^3}p(\cos y_1,\cos y_2,\cos(y_1+y_2))\dd y=\frac{11}{20}-\frac{15}{2}+\frac{1}{2}+\frac{3}{4}+2+\frac{1}{2}=-\frac{16}{5}<-\frac{3}{2}.$$
Testing with directions $B=(a,b,0)$, where $ab\le 0$, amounts to only consider the partial Hessian w.r.t.~$(x,y)$. Differentiation gives 
\begin{align*}
        p_{xx}&
=
8\left(x-\frac34y\right)^2
+\frac32y^2
+8\left(z-\frac5{24}\right)^2
+\frac{19}{360}>0,\\
p_{yy}&
=
8\left(y-\frac34x\right)^2
+\frac32x^2
+8\left(z-\frac5{24}\right)^2
+\frac{19}{360}>0,\\
p_{xy}
&=-6(x-y)^2
-8\left(z+\frac{15}{8}\right)^2
-\frac{15}{8}<0,
\end{align*}hence
$$D^2p(x,y,z)[(a,b,0),(a,b,0)]
=p_{xx}a^2+2p_{xy}ab+p_{yy}b^2
\ge 0.$$
It remains to consider directions $B=(a,-a,c)$.Introducing trace-deviatoric coordinates $s=x+y$ and $d=x-y$. Then $p(x,y,z)=\hat{p}(s,d,z)$ and
\begin{align*}
    \hat{p}(s,d,z)&=-\frac{37}{5}s^2+\frac{38}{5}d^2+\frac7{10}z^2
-\frac{25}{3}s^2z+\frac{20}{3}d^2z+\frac32z^3\\
&+\frac1{48}s^4+\frac18s^2d^2+\frac{25}{48}d^4
+4d^2z^2+\frac43z^4.
\end{align*}
Along the symmetric rank-one lines with direction $(a,-a,c)$, the transformed
coordinates satisfy
$$s(t)=s,\qquad
d(t)=d+2at,\qquad
z(t)=z+ct.$$
Thus
\begin{equation*}
D^2p(x,y,z)[(a,-a,c),(a,-a,c)]
=
\begin{pmatrix}2a&c\end{pmatrix}
D^2_{(d,z)}\hat{p}(s,d,z)
\begin{pmatrix}2a\\c\end{pmatrix}.
\end{equation*}
Differentiation gives
$$D^2_{(d,z)}\hat{p}(s,d,z)=
\begin{pmatrix}
\alpha(z)+\frac{25}{4}d^2+\frac14s^2
&
d\left(\frac{40}{3}+16z\right)
\\
d\left(\frac{40}{3}+16z\right)
&
\gamma(z)+8d^2
\end{pmatrix}$$
where
\begin{align*}
\alpha(z)
&=8\left(z+\frac56\right)^2+\frac{434}{45}>0,
\\
\gamma(z)
&=16\left(z+\frac9{32}\right)^2+\frac{43}{320}>0.
\end{align*}
For positive semi-definiteness, it suffices to consider the case $s=0$ and show $\det D^2_{(d,z)}\hat{p}(0,d,z)\ge 0$ by Sylvester's criterion. This expands as
\begin{align*}
  \det D^2_{(d,z)}\hat{p}(0,d,z)
&=
\left(\alpha(z)+\frac{25}{4}d^2\right)
\left(\gamma(z)+8d^2\right)
-d^2\left(\frac{40}{3}+16z\right)^2\\
&=50d^4+\beta(z)d^2+\alpha(z)\gamma(z),
\end{align*}
where 
$$\beta(z)=-92z^2-\frac{1055}{4}z-\frac{8537}{180}.$$
Note that $\beta$ is concave and $\beta(-2),\beta(-\tfrac{1}{5})>0$. Let us therefore consider the four regions for $z$
    $$(-\infty,-2), \quad [-2,-\tfrac{1}{5}],  \quad (-\tfrac{1}{5},0], \quad (0,\infty).$$
    Within $[-2,-\tfrac{1}{5}]$ the term $\beta(z)$ is positive, hence the determinant is non-negative. For the other three strips, we rewrite 
    \begin{align*}
  \det D^2_{(d,z)}\hat{p}(0,d,z)
&=50\left(d^2+\frac{\beta(z)}{100}\right)^2
+\frac{\varrho(z)}{200}.
\end{align*}It suffices to show that
\begin{align*}
 \varrho(z)
 &:=
 200A(z)\gamma(z)-\beta(z)^2\\
 &=
 17136z^4+\frac{25610}{3}z^3
 -\frac{2455757}{720}z^2
 +\frac{437413}{72}z
 +\frac{65014031}{32400}\ge 0.
\end{align*}To do so, we take the coefficient bounds
\begin{align*}
    17136&>17000,\quad
0<\frac{25610}{3}<8600, \quad 0<\frac{2455757}{720}<3500,\\ 
&0<\frac{437413}{72}<6100,\quad
\frac{65014031}{32400}>2000,
\end{align*}
and make the case distinction
\begin{align*}
z\ge0:\quad
\varrho(z)&\ge17000z^4-3500z^2+2000\\
&\ge2000-\frac{3500^2}{4\cdot17000}>0,
\\
-\tfrac15\le z\le0:\quad
\varrho(z)&\ge2000-\frac{8600}{125}
-\frac{3500}{25}-\frac{6100}{5}
=\frac{2856}{5}>0,
\\
z<-2,\quad w=-z:\quad
\varrho(z)&\ge
\left(17000-4300-875-\frac{6100}{8}\right)w^4+2000\\
&=\frac{22125}{2}w^4+2000>0.
\end{align*}
\end{proof}

We now localize this polynomial and extend the result to a function that is SRC, analogous to the procedure in Thm.~\ref{Thm: Scheme for transfering RCS to SRC}. In order to prove HQC, we use the heat-flow argument recently considered in 
\cite{bruno2026quasiconvexitydacorognamarcellinienergy}. 

\begin{proof}[Proof of Thm.~\ref{Thm: No SQC}]
Let $\mathbb L$ be as in \eqref{eq: Pati subspace}.
    Let $P$ be the orthogonal projection onto $\mathbb{L}$ and $Q:=\id-P$. Furthermore, let us take $u\in C^\infty(\mathbb T^3,\R^3)$
    $$u(x,y,z)=\frac{1}{2\pi}\bigl(\sin(2\pi x),\sin(2\pi y),\sin(2\pi(x+y))\bigr)^T.$$
    Then
    $$\varepsilon(u)=S\bigl(\cos(2\pi x),\cos(2\pi y),\cos(2\pi(x+y))\bigr),$$
    and, by identifying $\mathbb{L}\cong \R^3$, as well as using $p$ from Lem.~\ref{Lem: quartic Pati polynomial},
    $$\int_{\mathbb T^3}p(\varepsilon(u))\dd x<-\frac{3}{2}, \quad \int_{\mathbb T^3}\vert \varepsilon(u)\vert^2\dd x=\frac{3}{2}, \quad \Vert \varepsilon(u)\Vert_{L^\infty}=\sqrt{3}. $$
\textit{1.~Localization and extension: }Reusing the localization from Thm.~\ref{Thm: Scheme for transfering RCS to SRC}, we choose $\zeta\in C^\infty([0,\infty),[0,1])$ such that $\zeta \equiv0$ on $[0,2]$ and $\zeta\equiv 1$ on $[\tfrac{5}{2},\infty)$. Let $\rho\colon [0,\infty)\to \R$ be its second antiderivative admitting $\rho(0)=0$ and $\rho'(0)=0$. 
We define a radially symmetric 
$h\colon \mathbb{L}\to \R, \ A\mapsto \rho(\vert A\vert)$, and 
one can show analogously that 
$$D^2 h(A)[B,B] \ge \frac{1}{6}\vert B\vert^2$$
holds on $\mathbb{L}\setminus \overline{B(0,3)}$. Moreover, by choosing a smooth cutoff function 
 $\chi\in C^\infty(\mathbb{L},[0,1])$ such that $\chi\equiv 1$ on $\overline{B(0,3)}$ and $\chi\equiv 0$ on $\mathbb{L} \setminus B(0,4)$, we can find a constant $C_0\ge 0$ such that
$$D^2 (\chi p)(A)[B,B]\ge -C_0\vert B\vert^2$$
for every $A,B\in \mathbb{L}$. 
We then define 
$$g(A):=\chi(A)p(A) + \vert A\vert^2 + 6C_0h(A) $$
    for $A\in \mathbb{L}$. Thus, $g$ satisfies
    \begin{align}\label{eq: strong src for g on L}
         D^2g(A)[B,B]\ge2\vert B\vert^2
    \end{align}for all $A,B\in \mathbb{L}$, $B$ being compatible. In particular, $g$ is SRC in $\mathbb{L}$, and
    \begin{equation}
        g(0)=0 \quad \text{and}\quad \int_{\mathbb T^3}g(\varepsilon(u))\dd x<-\frac{3}{2}+\frac{3}{2}=0. \label{eq: g localization of p}
    \end{equation}
Set
$$C_g:=\sup_{A\in \mathbb{L}}\sup_{B\in \mathbb{L}\setminus\{0\}}\frac{\vert D^2g(A)[B,B]\vert }{\vert B\vert^2},$$
    which is finite since $g$ has at most quadratic growth and bounded Hessian outside of $B(0,3)$. 
 Let $\mathbb X$ be the set of compatible unit-norm matrices. By continuity of the distance and compactness of $\mathbb X$, there exists a $\delta>0$ such that for all $B\in \mathbb X$ admitting $\vert QB\vert< \delta $. It follows that 
 $$\operatorname{dist}(PB,\mathbb X\cap \mathbb{L})<\frac{1}{2C_g}.$$
  In other words, for every such $B$ there is a $B_0\in \mathbb X\cap \mathbb{L}$ such that
  $\vert PB-B_0\vert \le \tfrac{1}{2C_g}$. Thus, for every $A\in \mathbb{L}$, we infer by polarization, $\vert PB\vert\le 1$, and \eqref{eq: strong src for g on L}
  \begin{align*}
      D^2g(A)[PB,PB]\ge D^2g(A)[B_0,B_0]-C_g\vert PB-B_0\vert(\vert PB\vert+\vert B_0\vert)\ge 1.
  \end{align*}
  Let us now define the extension $f_3$ of $g$ via
$$f_3\colon \Symthree\to \R, \ A\mapsto g(PA)+\lambda \vert QA\vert^2,$$
  where $\lambda\ge \tfrac{C_g+1}{2\delta^2}$. Then, for every $A\in \Symthree$ and every $B\in \mathbb X$, $\vert QB\vert\ge \delta$, there holds
  $$D^2f_3(A)[B,B]\ge -C_g+2\lambda \delta^2\ge 1.$$
Therefore, $f_3$ is SRC. 
    Note that $\varepsilon(u)(x)\in \mathbb{L}$ for all $x\in \mathbb T^3$, thus $f_3$ is not SQC at the origin. It remains to show that $f_3$ is HQV.\\
\textit{2.~Heat-flow argument: }Let $A\in \Symthree$ and $\psi\in C^\infty_{\mathrm{per}}(\mathbb T^3,\R)$ be given. Let $\psi_t:=e^{t\Delta}\psi$, i.e., $\psi_t$ is a solution for the heat equation
$$\partial_t \psi_t=\Delta \psi_t, \quad \psi_0=\psi.$$
    Set $H_t:=A+D^2\psi_t$, then $\partial_tH_t=\Delta H_t$ componentwise, and
    \begin{align*}
        \frac{d}{dt}\int_{\mathbb T^3}f_3(H_t)\dd x &=\int_{\mathbb T^3}Df_3(H_t):\Delta H_t\dd x \\
        &=-\int_{\mathbb T^3}\sum_{k=1}^3D^2f_3(H_t)[\partial_kH_t,\partial_kH_t]\dd x.
    \end{align*}
If we can show that the last integral is nonnegative, then $t\mapsto\int_{\mathbb T^3}f_3(H_t)\dd x$ is nonincreasing
w.r.t.~$t$. For a smooth periodic $\psi$, the Fourier representation of the periodic heat semigroup
(cf.~\cite[pp.~161–165]{hunter2001applied}), gives
$$
D^2e^{t\Delta}\psi\longrightarrow0
\quad\text{uniformly as }t\to\infty,$$
and
$$D^2e^{t\Delta}\psi\longrightarrow D^2\psi
\quad\text{uniformly as }t\downarrow0.$$
Thus, we conclude by the continuity of $f_3$ that  
$$\int_{\mathbb T^3}f_3(H_t)\dd x\overset{t\to \infty}{\longrightarrow} f_3(A) , \quad \int_{\mathbb T^3}f_3(H_t)\dd x\overset{t\to 0}{\longrightarrow}   \int_{\mathbb T^3}f_3(A+D^2\psi)\dd x.$$
In particular
$$ \int_{\mathbb T^3}f_3(A+D^2\psi)\dd x \ge f_3(A).$$
Lastly, it suffices to show that for the dissipation there holds 
$$\sum_{k=1}^3D^2f_3(H_t)[\partial_kH_t,\partial_kH_t]\ge 0. $$
\textit{3.~Nonnegative dissipation: }Let us write $X_{ijk}:=\partial_i\partial_j\partial_kv_t$, which is fully symmetric. We can find a $z_k\in \R^3$ such that $P (\partial_k H_t)=S(z_k)$ via
$$z_k:=(X_{11k},X_{22k},X_{13k}+X_{23k}).$$
Isolating the diagonal rank-one entries
$$d_1=(X_{111},0,0), \quad d_2=(0,X_{222},0), \quad d_3=0,$$
we can control the remainder terms $r_k=z_k-d_k$, exploiting the full symmetry of $X_{ijk}$, by the complementary projections $Q\partial_k H_t$. 
More precisely, if $Q\partial_k H_t=0$ for every $k$, then
$$X_{12k}=X_{33k}=0,\qquad X_{13k}=X_{23k}.$$
Full symmetry of $X$ gives
$$X_{112}=X_{122}=X_{123}
=X_{133}=X_{233}=X_{333}=0,$$
and the remaining identities imply $X_{113}=X_{223}=0.$
By linearity of the operators $(Q\partial_kH_t)_k$ and $(r_k)_k$, there exists a constant $C_Q>0$ such that
$$\sum_{k=1}^3 \vert r_k\vert^2\le C_Q\sum_{k=1}^3 \vert Q\partial_kH_t\vert^2. $$
Thus, by abuse of notation, we expand
\begin{align*}
   \sum_{k=1}^3 D^2f_3(H)[\partial_kH_t,\partial_kH_t]&=  \sum_{k=1}^3 D^2g(PH_t)[P\partial_kH_t,P\partial_kH_t]+ 2\lambda \sum_{k=1}^3 \vert Q\partial_kH_t\vert^2\\
   &= \sum_{k=1}^3 D^2g(PH_t)[z_k,z_k]+ 2\lambda \sum_{k=1}^3 \vert Q\partial_kH_t\vert^2.
\end{align*}
Bilinearity and Young's inequality together with \eqref{eq: strong src for g on L} guarantee
\begin{align*}
     D^2g(PH_t)[z_k,z_k]&= D^2g(PH_t)[d_k,d_k]+  D^2g(PH_t)[r_k,r_k]+2  D^2g(PH_t)[d_k,r_k]\\
    &\ge 2\vert d_k\vert^2-C_g\vert r_k\vert^2-C_g^2\vert r_k\vert^2-\vert d_k\vert^2\\
    &\ge \vert d_k\vert^2-(C_g+C_g^2)\vert r_k\vert^2.
\end{align*}
If $\lambda \ge \tfrac{C_Q}{2}(C_g+C_g^2)$, we conclude
$$\sum_{k=1}^3 D^2f_3(H_t)[\partial_kH_t,\partial_kH_t]\ge \sum_{k=1}^3\biggl(\vert d_k\vert^2+(2\lambda -C_Q(C_g+C_g^2))\vert Q\partial_kH_t\vert^2 \biggr)\ge 0.$$
In total, the claim follows by choosing
$$\lambda \ge \max\biggl\{ \tfrac{C_g+1}{2\delta^2} ,\tfrac{C_Q}{2}(C_g+C_g^2)\biggr\}.$$ 
\end{proof}

\section{Outlook}

We obtained three examples that disprove the general respective implications.
Furthermore, remark that our counterexample (Thm.~\ref{Thm: SQC no PCS}) and transfer scheme (Thm.~\ref{Thm: Scheme for transfering RCS to SRC}) are valid in any dimension. In particular, our diagram in Fig.~\ref{Fig 1: Diagram Semi-conveixty} holds for any dimension $n\ge 3$.

Concerning the two-dimensional Morrey problem, we note that our dimension-independent transfer scheme shows that a smooth integrand on $\Symtwo$ that is rank-one convex at symmetric matrices and fails to be Hessian quasiconvex,
would yield a symmetric rank-one convex integrand failing Hessian quasiconvexity, and therefore also failing symmetric quasiconvexity.
Composing the latter with the projection $\pi_S$ would provide
a counterexample to the classical Morrey problem on
$\mathbb R^{2\times2}$.
 It might therefore be worthwhile to investigate this restricted version of Morrey's problem in order to gain further insight into the classical one.

\printbibliography[heading=bibintoc, title={References}]

\end{document}